\documentclass[12pt]{amsart}
\usepackage{mathrsfs}
\usepackage{mathabx}
\usepackage{enumerate}
\usepackage{bbm}
\usepackage{mathrsfs, mathtools}
\usepackage{stmaryrd}
\usepackage{tikz}
\usetikzlibrary{cd,arrows,matrix,positioning}
\usepackage{amsmath,amssymb,amsfonts}
\usepackage[all,cmtip]{xy}
\usepackage{a4wide}
\usepackage[mathscr]{eucal}
\usepackage[toc,page]{appendix}
\usepackage{extarrow}
\usepackage{enumerate}
\usepackage{fontenc}
\usepackage[T1]{fontenc}
\usepackage{graphicx}
\usepackage[colorlinks=true, linkcolor=blue]{hyperref}
\usepackage{latexsym}
\usepackage{mathrsfs}
\usepackage{mathtext}
\usepackage{tikz}
\usepackage{textcomp}
\usepackage{upgreek}
\usepackage{xy}

\usetikzlibrary{arrows}
\usetikzlibrary{matrix}
\usetikzlibrary{shapes}
\usetikzlibrary{snakes}
\usetikzlibrary{matrix}

\DeclareMathOperator{\Bi}{\textup{B}}

\DeclareMathOperator{\End}{\textup{End}}

\DeclareMathOperator{\Hl}{\textup{H}}

\DeclareMathOperator{\Imm}{\textup{Im}}

\DeclareMathOperator{\et}{\textup{et}}

\usepackage[colorinlistoftodos]{todonotes}

\makeatletter
\newcommand*{\defeq}{\mathrel{\rlap{%
                     \raisebox{0.3ex}{$\m@th\cdot$}}%
                     \raisebox{-0.3ex}{$\m@th\cdot$}}%
                     =}
\makeatother

\theoremstyle{plain}
\newtheorem{thm}{Theorem}[section]
\newtheorem*{thm*}{Theorem}
\newtheorem{lem}[thm]{Lemma}
\newtheorem{cor}[thm]{Corollary}
\newtheorem{prop}[thm]{Proposition}

\newtheorem*{conj*}{Conjecture}

\newenvironment{thmcite}[1]
{\medskip \noindent {\bf Theorem (#1).}\em }{\rm \medskip}

\theoremstyle{definition}

\newtheorem{defn}[thm]{Definition} %\setcounter{defn}{0}
\newtheorem{ex}[thm]{Example}

\newtheorem{rmk}[thm]{Remark}

\numberwithin{thm}{section}
\newcounter{x}

\newcommand{\red}{{\rm red}}

\newcommand{\perf}{{\rm perf}}
\usepackage{comment}

\newcommand{\Pic}{{\rm Pic}}

\newcommand{\Hom}{{\rm Hom}}

\newcommand{\Spec}{{\rm Spec \,}}

\newcommand{\Gal}{{\rm Gal}}

\newcommand{\PI}{{\rm PI}}

\newcommand{\sD}{{\mathcal D}}

\newcommand{\sN}{{\mathcal N}}
\newcommand{\sO}{{\mathcal O}}
\newcommand{\sP}{{\mathcal P}}
\newcommand{\sQ}{{\mathcal Q}}

\newcommand{\sU}{{\mathcal U}}
\newcommand{\sV}{{\mathcal V}}

\newcommand{\sX}{{\mathcal X}}
\newcommand{\sY}{{\mathcal Y}}

\newcommand{\A}{{\mathbb A}}

\newcommand{\C}{{\mathbb C}}

\newcommand{\E}{{\mathbb E}}
\newcommand{\F}{{\mathbb F}}

\newcommand{\N}{{\mathbb N}}

\newcommand{\Q}{{\mathbb Q}}

\newcommand{\Z}{{\mathbb Z}}

\newcommand{\NN}{\textup{N}}
\newcommand{\LL}{\textup{L}}

\newcommand{\Vect}{\text{\sf Vect}}
\newcommand{\Rep}{\text{\sf Rep}\,}
\newcommand{\id}{{\rm id\hspace{.1ex}}}

\renewcommand{\et}{\textup{\'et}}

\theoremstyle{plain}
\newtheorem{thmI}{Theorem}
\newtheorem{thmII}{Theorem}
\newtheorem{conjI}{Conjecture}

\newtheorem{conjII}{Conjecture}

\begin{document}

\title{The arithmetic local Nori fundamental group}

\author{Matthieu Romagny, Fabio Tonini, Lei Zhang }
\address{ Matthieu Romagny\\
    Institut de Recherche Math\'ematique de Rennes, Universit\'e Rennes~1, Campus de
Beaulieu, 35042 Rennes Cedex, France}
\email{matthieu.romagny@univ-rennes1.fr }

\address{ Fabio Tonini\\ Dipartimento di Matematica e
    Informatica `Ulisse Dini', Universit\'a degli Studi di
    Firenze, Viale Morgagni, 67/a, Florence 50134, italy.
   }
\email{fabio.tonini@unifi.it}
 \address{ Lei Zhang\\
     Department of Mathematics, The Chinese University of Hong
Kong, Shatin, New Territories, Hongkong}
\email{lzhang@math.cuhk.edu.hk}

\thanks{This work was supported by  the Research Grants Council
    (RGC) of the Hongkong SAR China (Project No. CUHK 14301019). The second author was supported by GNSAGA of INdAM. 
The first author was supported by
the Centre Henri Lebesgue, program ANR-11-LABX-0020-01 and would like to
thank the executive and administrative staff of IRMAR and of the Centre
Henri Lebesgue for creating an attractive mathematical environment.}

\date{\today}

\global\long\def\A{\mathbb{A}}

\global\long\def\Ab{(\textup{Ab})}

\global\long\def\C{\mathbb{C}}

\global\long\def\Cat{(\textup{cat})}

\global\long\def\Di#1{\textup{D}(#1)}

\global\long\def\E{\mathcal{E}}

\global\long\def\F{\mathbb{F}}

\global\long\def\GCov{G\textup{-Cov}}

\global\long\def\Gcat{(\textup{Galois cat})}

\global\long\def\Gfsets#1{#1\textup{-fsets}}

\global\long\def\Gm{\mathbb{G}_{m}}

\global\long\def\GrCov#1{\textup{D}(#1)\textup{-Cov}}

\global\long\def\Grp{(\textup{Grps})}

\global\long\def\Gsets#1{(#1\textup{-sets})}

\global\long\def\HCov{H\textup{-Cov}}

\global\long\def\MCov{\textup{D}(M)\textup{-Cov}}

\global\long\def\MHilb{M\textup{-Hilb}}

\global\long\def\N{\mathbb{N}}

\global\long\def\PGor{\textup{PGor}}

\global\long\def\PGrp{(\textup{Profinite Grp})}

\global\long\def\PP{\mathbb{P}}

\global\long\def\Pj{\mathbb{P}}

\global\long\def\Q{\mathbb{Q}}

\global\long\def\RCov#1{#1\textup{-Cov}}

\global\long\def\RR{\mathbb{R}}

\global\long\def\Sch{\textup{Sch}}

\global\long\def\WW{\textup{W}}

\global\long\def\Z{\mathbb{Z}}

\global\long\def\acts{\curvearrowright}

\global\long\def\alA{\mathscr{A}}

\global\long\def\alB{\mathscr{B}}

\global\long\def\arr{\longrightarrow}

\global\long\def\arrdi#1{\xlongrightarrow{#1}}

\global\long\def\catC{\mathscr{C}}

\global\long\def\catD{\mathscr{D}}

\global\long\def\catF{\mathscr{F}}

\global\long\def\catG{\mathscr{G}}

\global\long\def\comma{,\ }

\global\long\def\covU{\mathcal{U}}

\global\long\def\covV{\mathcal{V}}

\global\long\def\covW{\mathcal{W}}

\global\long\def\duale#1{{#1}^{\vee}}

\global\long\def\fasc#1{\widetilde{#1}}

\global\long\def\fsets{(\textup{f-sets})}

\global\long\def\iL{r\mathscr{L}}

\global\long\def\id{\textup{id}}

\global\long\def\la{\langle}

\global\long\def\odi#1{\mathcal{O}_{#1}}

\global\long\def\ra{\rangle}

\global\long\def\set{(\textup{Sets})}

\global\long\def\sets{(\textup{Sets})}

\global\long\def\shA{\mathcal{A}}

\global\long\def\shB{\mathcal{B}}

\global\long\def\shC{\mathcal{C}}

\global\long\def\shD{\mathcal{D}}

\global\long\def\shE{\mathcal{E}}

\global\long\def\shF{\mathcal{F}}

\global\long\def\shG{\mathcal{G}}

\global\long\def\shH{\mathcal{H}}

\global\long\def\shI{\mathcal{I}}

\global\long\def\shJ{\mathcal{J}}

\global\long\def\shK{\mathcal{K}}

\global\long\def\shL{\mathcal{L}}

\global\long\def\shM{\mathcal{M}}

\global\long\def\shN{\mathcal{N}}

\global\long\def\shO{\mathcal{O}}

\global\long\def\shP{\mathcal{P}}

\global\long\def\shQ{\mathcal{Q}}

\global\long\def\shR{\mathcal{R}}

\global\long\def\shS{\mathcal{S}}

\global\long\def\shT{\mathcal{T}}

\global\long\def\shU{\mathcal{U}}

\global\long\def\shV{\mathcal{V}}

\global\long\def\shW{\mathcal{W}}

\global\long\def\shX{\mathcal{X}}

\global\long\def\shY{\mathcal{Y}}

\global\long\def\shZ{\mathcal{Z}}

\global\long\def\st{\ | \ }

\global\long\def\stA{\mathcal{A}}

\global\long\def\stB{\mathcal{B}}

\global\long\def\stC{\mathcal{C}}

\global\long\def\stD{\mathcal{D}}

\global\long\def\stE{\mathcal{E}}

\global\long\def\stF{\mathcal{F}}

\global\long\def\stG{\mathcal{G}}

\global\long\def\stH{\mathcal{H}}

\global\long\def\stI{\mathcal{I}}

\global\long\def\stJ{\mathcal{J}}

\global\long\def\stK{\mathcal{K}}

\global\long\def\stL{\mathcal{L}}

\global\long\def\stM{\mathcal{M}}

\global\long\def\stN{\mathcal{N}}

\global\long\def\stO{\mathcal{O}}

\global\long\def\stP{\mathcal{P}}

\global\long\def\stQ{\mathcal{Q}}

\global\long\def\stR{\mathcal{R}}

\global\long\def\stS{\mathcal{S}}

\global\long\def\stT{\mathcal{T}}

\global\long\def\stU{\mathcal{U}}

\global\long\def\stV{\mathcal{V}}

\global\long\def\stW{\mathcal{W}}

\global\long\def\stX{\mathcal{X}}

\global\long\def\stY{\mathcal{Y}}

\global\long\def\stZ{\mathcal{Z}}

\global\long\def\then{\ \Longrightarrow\ }

\global\long\def\L{\textup{L}}

\global\long\def\l{\textup{l}}

% \makeatletter 
% \providecommand\@dotsep{5} 
% \makeatother 
% \listoftodos\relax

\begin{abstract}
In this paper we introduce the local Nori fundamental group scheme
of a reduced scheme or algebraic stack over a perfect field $k$.
We give particular attention to the case of fields: to any field
extension $K/k$ we attach a pro-local group scheme over~$k$. We show
how this group has many analogies, but also some crucial differences,
with the absolute Galois group. We propose two conjectures,
analogous to the classical Neukirch-Uchida Theorem and Abhyankar
Conjecture, providing some evidence in their favor.
Finally we show that the local fundamental group of a normal variety
is a quotient of the local fundamental group of an open, of its generic
point (as it happens for the \'etale fundamental group) and even
of any smooth neighborhood.
\end{abstract}

\setcounter{section}{0}
\maketitle
\section*{Introduction}
Let $k$ be a perfect field of characteristic $p\ge 0$.
Given a $k$-variety $X$ with a rational
point $x\in X(k)$, Nori constructed a profinite group scheme
$\pi^\NN(X,x)$ satisfying the following universal property: there
exists a natural bijection
\[
\Hom_k(\pi^\NN(X,x),G)\arr
\{ \text{pointed } G\text{-torsors } (P,p)\to (X,x)\}
\]
for all (pro-)finite group schemes $G$ over $k$. This was later called
the {\em Nori fundamental group of $X$ at $x$}. Since every
finite group scheme is an extension of an \'etale group scheme by a
local one (recall that a group scheme is called {\em local}, or
{\em infinitesimal}, if it is finite and connected), it is also
interesting to focus on both subclasses. One finds that the
maximal pro-\'etale quotient $\pi^{\NN,\et}(X,x)$ and the maximal
pro-local quotient $\pi^\L(X,x)$ of Nori's fundamental group satisfy
the universal properties restricted to the subclasses.

In the same way as Grothendieck's \'etale fundamental group ``parametrizes''
Galois \'etale covers, Nori's fundamental group ``parametrizes'' pointed
torsors under finite group schemes. One important difference is the
necessity to label the torsors with a base point in Nori's version.
This has regrettable consequences: for instance $\pi^\NN(\Spec k,x)=1$
even if $k$ is not algebraically closed. One of the purposes of the
present paper is to remove the need for a rational point. For any reduced
scheme or algebraic stack $X$ over $k$ we will define the {\em local Nori
fundamental group scheme $\pi^\L(X/k)$}, a pro-local group scheme over
$k$ with bijections functorial in the (pro)-local group scheme $G$:
\[
\Hom_k(\pi^\L(X/k),G)\arr \{ G\text{-torsors } P \to X\}.
\]
This definition makes sense because the category of torsors under
{\em local} group schemes is equivalent to a set, that is it is a
groupoid with only identities as automorphisms, unlike what happens
for general finite group schemes.

The emancipation from the need of a rational point in the theory of Nori's
fundamental group was also addressed by Borne and Vistoli \cite{BV15}.
Their idea was to use finite gerbes instead of finite group schemes.
The so-called {\em Nori fundamental gerbe} thus obtained recovers the
Nori fundamental group as soon as one fixes a rational point.
In \cite[\S 7]{TZ19} the same theory was worked out for local gerbes,
obtaining the local Nori fundamental gerbe (see~\ref{local nori gerbe}).
It turns out that the gerbe ``is exactly'' the local Nori fundamental group
without the need of any point, because a local gerbe over a perfect field
is uniquely and therefore canonically neutral. This is actually how we
come up with the above definition (see~\ref{profinite gerbe}). Moreover,
as a consequence of \cite{TZ19} there is an explicit Tannakian description
of the representation category $\Rep\pi^\L(X/k)$ in terms of vector
bundles on $X$ (see~\ref{Tannakian interpretation of the local nori gerbe}).

\begin{center} $\therefore$ \end{center}

As an evidence that the local Nori fundamental group has valuable
arithmetic content, we shall see that it is an extremely interesting
object when $X$ is the spectrum of a field $K$. In this case, we use
the simplified notation
\[
\pi^\L(K/k)\defeq\pi^\L(\Spec K/k).
\]
Due to the parallel between the local Nori fundamental group of a field $K$
and its absolute Galois group, it is desirable to see to what extent
their behaviours resemble or differ. We approach the question from
three viewpoints: the Galois correspondence, the anabelian philosophy,
and the inverse Galois problem. In the rest of the introduction, we
present our findings: two main results and two conjectures with piece
of evidence supporting them. Our first main theorem is a part of the Galois
correspondence.

\begin{thmI} \label{The Main Theorem}
Let $k$ be a perfect field and $K/k$ a field extension.
Denote by $\PI(K)$ the totally ordered set of purely inseparable
extensions of $K$. The mapping
   \[
  \begin{tikzcd}[column sep=20,
    ,row sep = 0ex
    ,/tikz/column 1/.append style={anchor=base east}
    ,/tikz/column 2/.append style={anchor=base west}
    ]
\PI(K) \ar[r] & \{\textup{subgroups of }\pi^\LL(K/k)\} \\
E/K \ar[r,mapsto] & (\pi^\LL(E/k)\to \pi^\L(K/k)).
  \end{tikzcd}
  \]
is well-defined and an order-reversing embedding, that is
$\pi^\L(E_1/k)\subseteq \pi^\L(E_2/k) \subseteq \pi^\L(K/k)$
if and only if $K\subseteq E_2\subseteq E_1$
for all $E_1,E_2\in \PI(K)$.
\end{thmI}

It would be interesting to understand the relation between the
above embedding and previous classifications of purely
inseparable extensions, starting from the Jacobson
correspondence and its early developments (see \cite{Ja44},
\cite{Ge64}, \cite{GZ70}, \cite{Ch71}, \cite{He71}, \cite{Mo75})
until very recent generalizations (see \cite{Ba18} and \cite{BW20}).

We also show that the role played in Galois Theory by separably closed
fields is played in the theory of the local Nori fundamental group by
perfect fields, in the precise sense that~$K$ is perfect if and
only if $\pi^\LL(K/k)=1$ (see~\ref{perfect means group zero}).
However, exceeding enthusiasm for the desired analogy should not reign.
Indeed, we prove that the map in Theorem~\ref{The Main Theorem} is highly
non surjective unless $K$ is perfect. Moreover,
for a nontrivial finite purely inseparable extension $L/K$ the subgroup
$\pi^\LL(L/k)$ never has finite index in $\pi^\LL(K/k)$
(see~\ref{counterexamples}).

Nevertheless, we expect $\pi^\L(E/k)$ to carry a lot of information
on the field $E$ and in some cases to be able to recover $E$. We
propose the following anabelian-style conjecture.

\begin{conjI}\label{NU conjecture}
 Let $k$ be a perfect field and $K,E$ finitely generated field
extensions of $k$. Assume that $K$ and $E$ are not finite over $k$.
Then
 \[
 \pi^\L(K/k) \simeq \pi^\L(E/k) \then K\simeq E
 \]
(isomorphism of $k$-group schemes on the left, isomorphism
of $k$-extensions on the right).
\end{conjI}

This is the local analogue of (a part of) the Neukirch-Uchida Theorem:

\begin{thmcite}{\cite{Uch77}, \cite{Pop1}, \cite{Pop2}, \cite{Mo99}}
If $K,E$ are number fields, then
\[
\Gal(K)\simeq \Gal(E)\then K\simeq E.
\]
If $K,E$ are infinite and finitely generated fields over $\F_p$ then
\[
\Gal(K)\simeq \Gal(E) \then K^\perf\simeq E^\perf
\]
where $(-)^\perf$ denotes the perfect closure functor of fields.
\end{thmcite}

Theorem~\ref{The Main Theorem} aims to be a baby case of
Conjecture~\ref{NU conjecture}.
%The reason why we do not use the
%separable closure in Conjecture \ref{NU conjecture} is that actually
%the local fundamental group is sensible to separable extensions.
More generally we show that, in the hypothesis of
Conjecture~\ref{NU conjecture}, a $k$-map $K\to E$ induces an
isomorphism $\pi^\L(E/k) \to \pi^\L(K/k)$ if and only if it is an isomorphism
(see~\ref{isom fund group}). The map $\pi^\L(E/k)\to \pi^\L(K/k)$ is actually
a quotient for separable extensions $E/K$ (see~\ref{surjective fund group}).
The reason why we expect Conjecture~\ref{NU conjecture} to be true is
that the explicit description of $\Rep \pi^\L(K/k)$ seems to tell a lot
about the arithmetic of~$K$. For instance, we construct isomorphisms
of groups:
\[
\Hom(\pi^\L(K/k),\Gm)\simeq (K^\perf)^*/K^*
\quad\text{and}\quad \Hom(\pi^\L(K/k),\mathbb{G}_a)\simeq K^\perf/K.
\]
Multiplicative and additive structures on $K^\perf$ are determined
by $\pi^\L(K/k)$, yet we still see no way to relate them for the
moment. The above
groups can be described in terms of one-dimensional representations
of $\pi^\L(K/k)$ and extensions of the trivial representation
respectively. We therefore expect that crucial information is contained
in higher dimensional representations.

Finally, in the spirit of the inverse Galois problem, we consider the
fundamental example given by the field of rational functions in one variable,
and we state the following conjecture.

\begin{conjII}\label{A conjecture}
 Let $k$ be a perfect field. Then any local group scheme $G$ over $k$ is a quotient of $\pi^\L(k(t)/k)$, that is there exists a $G$-torsor $P\to \Spec k(t)$ which is not induced by a torsor under a strict subgroup of $G$.
\end{conjII}

This is analogous to the ''generic Abhyankar conjecture''
%\cite[Conjecture 1.1]{Har94}
for the absolute Galois group:

\begin{thmcite}{\cite{Har94}}
 Let $k$ be an algebraically closed field. Then any finite group~$G$
is a quotient of $\Gal(k(t))$, that is there exists a Galois extension
of $k(t)$ with group~$G$.
\end{thmcite}

We call Conjecture \ref{A conjecture} the local Abhyankar conjecture.
A very similar, stronger conjecture was stated in
\cite[Question 1.1]{ot18} and proved in several cases in the same paper
and in \cite{ot19}, \cite{OTZ20}. The setting is slightly different:
for an affine $k$-curve $U$, Otabe \cite[Question 1.1]{ot18} predicts
which local group schemes over $k$ appear as a quotient of $\pi^\L(U/k)$,
just like the general Abhyankar conjecture predicts which finite groups
occur as a quotient of the \'etale fundamental group $\pi_1(U,u)$.
Our contribution here is in proving that the relation between the global
and generic local Abhyankar conjectures is similar to that in the
non-local case. Namely, the generic local fundamental group surjects
onto the global one. We prove more generally the following result:

\begin{thmII}\label{main theorem II}
 Let $\stX$ be a normal, quasi-separated and irreducible algebraic stack over $k$ and
$\stV\to \stX$ a map from a reduced algebraic stack. If $\stV\to \stX$ has non empty reduced geometric generic fiber then the map
 \[
 \pi^\L(\stV/k)\to \pi^\L(\stX/k)
 \]
 is surjective. This is the case, for instance, if $\stV\to \stX$ is either
\begin{enumerate}
\item[{\rm (1)}] flat, geometrically reduced and has an open image
(e.g.  an open embedding);
\item[{\rm (2)}] it exhibits $\stV$ as a generic point of a smooth atlas of $\sX$.
 \end{enumerate}
\end{thmII}

In particular, all groups considered in \cite{ot18} and \cite{OTZ20} satisfy Conjecture \ref{A conjecture}.
Finally we show that Conjecture \ref{A conjecture} implies the same result for many other fields. Indeed we show that $\pi^\L(k((t))/k)\to \pi^\L(k(t)/k)$ is surjective (see \ref{series surjective}) and that, if $K$ is any finitely generated field extension of $k$ of positive transcendence degree, there exists an indeterminate $t\in K$ such that also $\pi^\L(K/k)\to \pi^\L(k(t)/k)$ is surjective (see \ref{one variable is enough}).

The paper is divided as follows. In Section~\ref{local Nori over a field}
we define the local Nori fundamental group for a field using Tannaka
Theory, and we prove that Theorem~\ref{The Main Theorem} holds for this
group scheme. In Section~\ref{local Nori in general} we introduce the
general local Nori fundamental group via the local Nori fundamental
gerbe, and we prove the connection with torsors via the universal
property as stated in the beginning in the introduction. Finally in
Section~\ref{section:generic surjectivity} we discuss surjectivity results about
the local fundamental group, and we prove Theorem~\ref{main theorem II}.

\bigskip

\noindent {\bf Acknowledgement.}
We would like to thank  D. Tossici, S. Otabe and A. Vistoli for helpful conversations and suggestions received.

\section{The local Nori fundamental group of a field}
\label{local Nori over a field}

Let $k$ be a perfect field of positive characteristic $p$.
Given a field $K$ we will denote by $K^\perf$ its perfect closure
and by $\Vect(K)$ the category of finite dimensional $K$-vector
spaces. The aim of this section is to give a direct definition of
the local fundamental group scheme of a field extension $K/k$ that
avoids torsors. Then we prove that with this definition, the conclusion
of Theorem~\ref{The Main Theorem} holds. The equivalence between the
two definitions, namely the verification of the universal property,
is proved later (see \ref{prop: universal property local fund group}).

\begin{defn}\label{Local Tanna}
Let $K/k$ be a field extension. We define $\sD_{\infty}(K/k)$
as the category whose objects are triples $(V,W,\psi)$, where
$V\in\Vect(K)$, $W\in\Vect(k)$ and
\[\psi:K^{\perf}\otimes_KV\to K^{\perf}\otimes_kW\]
is a $K^\perf$-linear isomorphism. An arrow $(V,W,\psi)\to (V',W',\psi')$
in $\sD_{\infty}(K/k)$ is a pair $(a,b)$ composed of a $K$-linear
map $a\colon V\to V'$ and a $k$-linear map $b\colon W\to W'$ which
are compatible with $\psi$ and $\psi'$.
\end{defn}

The category $\shD_\infty(K/k)$ with its natural tensor product and $k$-linear structure is a neutral $k$-Tannakian category with the forgetful functor $\shD_\infty(K/k)\arr \Vect(k)$ as the fiber functor. We define the \emph{local Nori fundamental group} $\pi^\LL(K/k)$ of $K/k$ as the Tannakian group scheme associated with $\shD_\infty(K/k)$.

\begin{lem}\label{injectivity}
If $L/K$ is a purely inseparable extension of fields over $k$ then the induced map of group schemes $\pi^\LL(L/k)\to\pi^\LL(K/k)$ is a closed embedding.
\end{lem}
\begin{proof} By \cite[Prop. 2.21 (b)]{DM82} it is enough to show that the
    pullback functor 
    $$\sD_{\infty}(K/k)\to \sD_{\infty}(L/k)$$ is essentially surjective. 
    Let $(V,W,\phi)\in\shD_\infty(L/k)$. Clearly there is an isomorphism
$(V,W,\phi)\cong(L^{\oplus m},k^{\oplus m},\varphi)\in\shD_\infty(L/k)$.
    Now consider the isomorphism
\[\varphi\colon (L^\perf)^{\oplus m}\arr(L^\perf)^{\oplus m}.\]
    Since $L/K$ is purely inseparable, we can identify $K^\perf$ with $L^\perf$. In this case it is easy to see that 
    $(K^{\oplus m},k^{\oplus m},\varphi)\in\shD_\infty(K/k)$ is sent to 
    $(L^{\oplus m},k^{\oplus m},\varphi)$.
\end{proof}

\begin{prop}\label{picard of local gerbe}
Let $K/k$ be a field extension. Then there is a canonical isomorphism
\[
\Pic(\shD_\infty(K/k))\simeq \Hom_k(\pi^\LL(K/k),\Gm)\simeq (K^\perf)^*/K^*.
\]
\end{prop}
\begin{proof}
The first isomorphism exists since both sides are the group of isomorphism classes of $1$-dimensional representations of $\pi^\LL(K/k)$. There is a group homomorphism
   \[
  \begin{tikzcd}[column sep=20,
    ,row sep = 0ex
    ,/tikz/column 1/.append style={anchor=base east}
    ,/tikz/column 2/.append style={anchor=base west}
    ]
(K^\perf)^* \ar[r] & \Pic(\shD_\infty(K/k)) \\
\phi \ar[r,mapsto] & (K,k,\phi).
  \end{tikzcd}
  \]
It is easy to see that it is surjective and that its kernel is $K^*$.
%   We use notation from \ref{Tannakian interpretation of the local nori gerbe}. We have
%  \[
%  \Pic(\Pi^\LL_\stX)\simeq \Pic(\shD_\infty)\simeq \varinjlim_n \Pic(\shD_n)
%  \]
%  Using that $k$ is perfect and that, if $\shL$ is an invertible sheaf on $\stX$, there is a canonical isomorphism $F^*\shL\simeq \shL^p$, where $F$ is the Frobenius of $\stX$, it is easy to check that the sequence
%    \[
%   \begin{tikzpicture}[xscale=2.5,yscale=-0.6]
%     \node (A0_2) at (2, 0) {$(\shF,V,\lambda)$};
%     \node (A0_3) at (3, 0) {$\shF$};
%     \node (A1_0) at (0, 1) {$\Hl^0(\odi\stX^*)$};
%     \node (A1_1) at (1, 1) {$\Hl^0(\odi\stX^*)$};
%     \node (A1_2) at (2, 1) {$\Pic(\shD_n)$};
%     \node (A1_3) at (3, 1) {$\Pic(\stX)[p^n]$};
%     \node (A1_4) at (4, 1) {$0$};
%     \node (A2_1) at (1, 2) {$\mu$};
%     \node (A2_2) at (2, 2) {$(\odi\stX,k,\mu)$};
%     \path (A0_2) edge [|->,gray]node [auto] {$\scriptstyle{}$} (A0_3);
%     \path (A1_0) edge [->]node [auto] {$\scriptstyle{p^n}$} (A1_1);
%     \path (A1_1) edge [->]node [auto] {$\scriptstyle{}$} (A1_2);
%     \path (A1_2) edge [->]node [auto] {$\scriptstyle{}$} (A1_3);
%     \path (A2_1) edge [|->,gray]node [auto] {$\scriptstyle{}$} (A2_2);
%     \path (A1_3) edge [->]node [auto] {$\scriptstyle{}$} (A1_4);
%   \end{tikzpicture}
%   \]
% is exact. Taking the limit we get the result.
\end{proof}

\begin{proof}[Proof of Theorem \ref{The Main Theorem}]
By \ref{injectivity} we see that the map in the statement is well defined. All elements of $\PI(K)$ has a unique embedding in $K^\perf$. Thus we have to prove that, if $K\subseteq L_1,L_2\subseteq K^\perf$ and $\pi^\LL(L_1/k)\subseteq \pi^\LL(L_2/k)\subseteq \pi^\LL(K/k)$, then $L_2\subseteq L_1$.

By \ref{picard of local gerbe} and the commutative diagram of group schemes
\[\begin{tikzpicture}[xscale=3.0,yscale=-1.4]
    \node (A0_0) at (0, 0) {$\pi^\LL(L_1/k)$};
    \node (A0_2) at (2, 0) {$\pi^\LL(L_2/k)$};
    \node (A1_1) at (1, 1) {$\pi^\LL(K/k)$};
%    \node (A1_2) at (2, 1) {$\pi_1^{\text{\'et}}(X,x)$};
    \path (A0_0) edge [->]node [auto] {$\scriptstyle{\phi}$} (A0_2);
    \path (A0_0) edge [->]node [below] {$\scriptstyle{s_{L_1}}$} (A1_1);
    \path (A0_2) edge [->]node [auto] {$\scriptstyle{s_{L_2}}$} (A1_1);
%    \path (A0_0) edge [->>,dashed]node [auto] {$\scriptstyle{d}$} (A1_0);
  \end{tikzpicture}
\]
we get a commutative diagram of abelian groups
\[\begin{tikzpicture}[xscale=3.0,yscale=-1.4]
    \node (A0_0) at (0, 0) {$(K^{\perf})^*/L_2^*$};
    \node (A0_2) at (2, 0) {$(K^{\perf})^*/L_1^*$};
    \node (A1_1) at (1, 1) {$(K^{\perf})^*/K^*$};
%    \node (A1_2) at (2, 1) {$\pi_1^{\text{\'et}}(X,x)$};
    \path  (A0_0) edge [->]node [above] {$\scriptstyle{\varphi}$} (A0_2);
    \path  (A1_1)edge [->]node [auto] {$\scriptstyle{a}$} (A0_0);
    \path (A1_1) edge [->]node [below] {$\scriptstyle{b}$} (A0_2);
%    \path (A0_0) edge [->>,dashed]node [auto] {$\scriptstyle{d}$} (A1_0);
  \end{tikzpicture}
\]
where $s_{L_i}$ is the inclusion $\pi^\LL(L_i/k)\subseteq \pi^\LL(K/k)$. In particular we conclude
that $a$ and $b$ are induced on the perfect closure by $K\subseteq L_2$ and 
$K\subseteq L_1$ respectively, that is $a$ and $b$ are induced by the identity map. 
Then it is clear that $b(L_2^*/K^*)=1$, as  $a(L_2^*/K^*)=1$. Thus we have $L_2^*/K^*\subseteq L_1^*/K^*$. This shows the inclusion $L_2\subseteq L_1$.
\end{proof}

\begin{cor}\label{perfect means group zero}
 Let $K/k$ be a field extension. Then $K$ is perfect if and only if $\pi^\LL(K/k)=1$.
\end{cor}
\begin{proof}
 If $\pi^\LL(K/k)=0$, then from Theorem \ref{The Main Theorem} we see
that $\PI(K)$ has just one element, that is $K$ is perfect. Now assume
that $K$ is perfect, that is $K^\perf=K$. If $(V,W,\phi)\in\shD_\infty(K/k)$,
then $\phi\colon V\arr W\otimes_k K$ is an isomorphism. It is easy to
see that 
 \[(V,W,\phi)\arrdi{(\phi,\id)}(W\otimes_k K,W,\id)\]
 is an isomorphism. This means that $\Vect(k)=\shD_\infty(K/k)$,
that is $\pi^\LL(K/k)=1$.
\end{proof}

\begin{lem}\label{maximal diag quotient}
 If $G$ is an affine group scheme over $k$ then there exists a canonical surjective map 
 $$
 G\arr \Di{\Hom(G,\Gm)}
 $$
 where $\Di -$ is the diagonalizable group over $k$ associated with an abelian group, 
 which is universal among all maps to diagonalizable group schemes. 
\end{lem}
\begin{proof}
There is a natural isomorphism
\[
\Hom(G,\Di H) \simeq \Hom(H,\Hom(G,\Gm))
\]
which implies the existence of a map
\[
\delta\colon G\to \Di{\Hom(G,\Gm)}
\]
universal for maps to diagonalizable group schemes. Since $\Imm(\delta)$ is diagonalizable and a closed subgroup of $\Di{\Hom(G,\Gm)}$, by the universal property it follows that $\delta$ is surjective.
%  Let $\shR$ be the category of surjective maps $G\twoheadrightarrow H$ such that $H$ is a diagonalizable group. 
%  It is easy to see that $\shR$ is a cofiltered category, 
%  because products and subgroups of diagonalizable groups are diagonalizable. 
%  The group scheme
%  $$
%  G'=\varprojlim_{(G\arr H)\in \shR}H
%  $$
%  is diagonalizable because it is a projective limit of diagonalizable groups.
%  Moreover, there is a map $G\arr G'$, which is surjective because $k[G']$ is just 
%  a union of sub Hopf algebras of $k[G]$ whose corresponding group schemes
%  are diagonalizable. It is also easily seen to be universal among maps from
%  $G$ to diagonalizable groups, as a sub-group scheme of a diagonalizable group scheme is still diagonalizable.
%  In particular $\Hom(G,\Gm)=\Hom(G',\Gm)$ and, since $G'$ is diagonalizable, $G'=\Di{\Hom(G',\Gm)}$.
\end{proof}

\begin{ex}\label{counterexamples}
Let $K$ be a field extension of $k$ which is not perfect. We claim that:
\begin{enumerate}
 \item there are subgroups $H$ of $\pi^\LL(K/k)$ which are not of the
form $\pi^\LL(L/k)$ for some purely inseparable extension $L/K$;
 \item if $L/K$ is a nontrivial purely inseparable extension, then the
quotient fpqc space $\pi^\LL(K/k)/\pi^\LL(L/k)$ is not finite over $k$.
\end{enumerate}
From \ref{picard of local gerbe} and \ref{maximal diag quotient} we obtain a canonical surjective map $\pi^\LL(K/k)\arr \Di{(K^\perf)^*/K^*}$ which is universal among maps to a diagonalizable group scheme. Given a purely inseparable extension $L/K$ we obtain a commutative diagram
  \[
  \begin{tikzpicture}[xscale=3.3,yscale=-1.2]
    \node (A0_0) at (0, 0) {$\pi^\LL(L/k)$};
    \node (A0_1) at (1, 0) {$\pi^\LL(K/k)$};
    \node (A0_2) at (2, 0) {$\pi^\LL(K/k)/\pi^\LL(L/k)$};
    \node (A1_0) at (0, 1) {$\Di{(L^\perf)^*/L^*}$};
    \node (A1_1) at (1, 1) {$\Di{(K^\perf)^*/K^*}$};
    \node (A1_2) at (2, 1) {$\Di{L^*/K^*}$};
    \path (A0_0) edge [right hook->]node [auto] {$\scriptstyle{}$} (A0_1);
    \path (A0_1) edge [->>]node [auto] {$\scriptstyle{}$} (A1_1);
    \path (A1_0) edge [right hook->]node [auto] {$\scriptstyle{}$} (A1_1);
    \path (A1_1) edge [->>]node [auto] {$\scriptstyle{}$} (A1_2);
    \path (A0_2) edge [->>]node [auto] {$\scriptstyle{}$} (A1_2);
    \path (A0_0) edge [->>]node [auto] {$\scriptstyle{}$} (A1_0);
    \path (A0_1) edge [->>]node [auto] {$\scriptstyle{}$} (A0_2);
  \end{tikzpicture}
  \]
In particular if $K^*\subseteq Q \subseteq (K^\perf)^*$ is a subgroup not
of the form $L^*$ for some purely inseparable extension $L$ of $K$ then 
the inverse image of $\Di{(K^\perf)^*/Q}\subseteq \Di{(K^\perf)^*/K^*}$
along $\pi^\LL(K/k)\arr \Di{(K^\perf)^*/K^*}$ cannot be a local fundamental 
group. For $(2)$ instead one just has to show that $L^*/K^*$ is not finitely 
generated. Indeed one observes that the vertical map on the right is 
faithfully flat because all other surjective maps are affine and faithfully 
flat (see \cite[\href{http://stacks.math.columbia.edu/tag/036J}{036J}]{SP21}). 

In order to have a concrete example for (1) and also show $(2)$ it is
enough to prove that, if $L=K[X]/(X^p-\lambda)$ with $\lambda\in K-K^p$,
then $L^*/K^*$ is not finitely generated. Since $L^*/K^*$ is an
$\F_p$-vector space it is enough to show that $L^*/K^*$ is infinite.
Set $v_n=1+\lambda^nX$ for $n\in \N$. We claim that all these elements
are different in $L^*/K^*$. If $v_m=v_n$ for some $m\neq n\in \N$
then a direct computation shows that $\lambda$ is a root of unity.
In particular it is algebraic and thus separable over $k$. In
conclusion $X$ is purely inseparable over the perfect field
$k(\lambda)$, from which we find the contradiction $X\in k(\lambda)\subseteq K$.  
\end{ex}

\section{The local Nori fundamental group in general}
\label{local Nori in general}

In this section we fix a base field $k$ of positive characteristic $p$.

\begin{defn}\label{local stuff}
  A group scheme $G$ over  $k$ is called \emph{local} if it is finite and connected.
  
  An affine gerbe $\Gamma$ over $k$ is called \emph{finite} (resp. finite and \emph{local}) if $\Gamma\times_k \overline k \simeq \Bi_{\overline k}G$, where $G$ is a finite (resp. finite and local) group scheme over $\overline k$.
  
  By a \textit{pro-local gerbe} (resp. \textit{pro-local group scheme}) over $k$ we mean a small cofiltered limit of finite and local gerbes (resp. group schemes) over $k$.
\end{defn}

\begin{defn}\label{local nori gerbe}
Let $\stX$ be an algebraic stack over $k$.
A \emph{local Nori fundamental gerbe} of $\stX/k$ is a pro-local
gerbe $\Pi$ over $k$ together with a morphism $\stX\arr \Pi$ such
that for all finite and local gerbes $\Gamma$ over $k$ the pullback functor
 \[
 \Hom_k(\Pi,\Gamma)\arr \Hom_k(\stX,\Gamma)
 \]
is an equivalence. If such a gerbe exists, it is unique and will
be denoted by $\Pi^\LL_{\stX/k}$.
\end{defn}

\begin{rmk}\label{existence of local gerbe}
Notice that a gerbe $\Gamma$ is finite (resp. finite and local) 
if and only it is finite (resp. finite and local) in the sense of 
\cite[Definition 3.1, p.~10]{TZ19} (resp. \cite[Definition 3.9, p.~12]{TZ19}). 
(See \cite[Proposition B.6, p.~41]{TZ19})

 By \cite[Theorem 7.1]{TZ19}, if $\stX$ is non-empty, reduced and
$\Hl^0(\odi\stX)$ does not contain nontrivial purely inseparable
field extensions of $k$, then the local Nori fundamental gerbe
$\Pi^\LL_{\stX/k}$ exists and coincides with the local Nori
fundamental gerbe considered
 in \cite[Definition 4.1, p.~12]{TZ19}. Moreover, the local Nori fundamental 
 gerbe is unique up to a unique isomorphism because pro-local gerbes are
 projective limits of finite and local gerbes. 
\end{rmk}

\begin{thm}{ \cite[Theorem 7.1]{TZ19} }
    \label{Tannakian interpretation of the local nori gerbe}
    Assume that $\stX$ is reduced and $\Hl^0(\odi\stX)$ does not
contain nontrivial purely inseparable field extensions of $k$.
Denote by $F\colon\stX\to\stX$ and $F_k:\Spec k\to\Spec k$
the absolute Frobenius morphisms.
    For $i\in \N$ denote by $\shD_i$ the category of triples 
    $(\shF,V,\lambda)$ where $\shF\in\Vect(\stX)$, $V\in\Vect(k)$ and 
    $\lambda\colon  F^{i*}\shF \to V\otimes_k\odi\stX$ is an isomorphism. 
    Then the category $\shD_i$ is  $k$-Tannakian with $k$-structure 
     $k\to \End_{\shD_i}(\odi\stX,k,\id)$, $x\mapsto (x,x^{p^i})$. 
    Moreover the functor
\[
\shD_i\arr\shD_{i+1}\comma (\shF,V,\lambda)\longmapsto (\shF,F_k^*V,F^*\lambda)
\]
is $k$-linear, monoidal and exact, and 
there is a natural equivalence of $k$-Tannakian categories:
\[
    \shD_\infty\coloneqq \varinjlim_{i\in\N}\shD_i \arrdi \simeq\Rep(\Pi^\LL_{\stX/k}).
\]
\end{thm}

\begin{prop}
With notation from \ref{Tannakian interpretation of the local nori gerbe} the functors
\[
\shD_i \to \shD_{i+1}
\]
and therefore the functors $\shD_i\to \shD_\infty$ are fully faithful.
In particular $\shD_i$ is a sub-Tannakian category of $\shD_j$ for $j>i$ and $j=\infty$.
\end{prop}
\begin{proof}
The last claim is a consequence of \cite[Remark B.7]{TZ19}:
the full faithfulness is enough to conclude that the
corresponding map on gerbes is a quotient.

 Consider $E=(\shE,V,\lambda)$, $E'=(\shE',V',\lambda')\in \shD_i$. We have
\[
\Hom(E,E')=\{(\alpha,\beta)\in \Hom(\shE,\shE')\times \Hom(V,V')\st \lambda'F^{i*}(\alpha)=(\beta\otimes\odi \stX)\lambda\}.
\]
The functor $\Phi\colon \shD_i\to \shD_{i+1}$ maps $(\E,V,\lambda)$
to $(\E,F_k^*V,F_k^*\lambda)$ and $(\alpha,\beta)\in \Hom(E,E')$
to $(\alpha,F_k^*\beta)$. Let $(\alpha,\delta)\in \Hom(\Phi(E),\Phi(E'))$. 
Fix isomorphisms $V\simeq k^n$ and $V'\simeq k^m$, so that
$F_k^*V\simeq k^n$ and $F_k^*V'\simeq k^m$. The map
$\delta\colon F_k^*V\to F_k^*V'$ is therefore a matrix
$\delta=(\delta_{ij})$ with $\delta_{ij}\in k$. Consider the map
\[
\beta=\lambda' F^{i*}(\alpha) \lambda^{-1}\colon V\otimes
\odi \stX\to V'\otimes \odi \stX.
\]
This is represented by a matrix $\beta=(\beta_{ij})$ with $\beta_{ij}\in
\Hl^0(\odi\stX)$. The hypothesis is that $\beta_{ij}^p=\delta_{ij}$.
Since $\stX$ is reduced, we have  $\beta_{ij}\in k$. Thus $(\alpha,\beta)\in \Hom(E,E')$ induces $(\alpha,\delta)$. 
\end{proof}

We now consider the case of a perfect field. In this case the local Nori gerbe exists for all reduced algebraic stacks.
% 
% \begin{rmk}
%  Assume $\stX$ is a reduced algebraic stack over a perfect field $k$. With notation from \ref{Tannakian interpretation of the local nori gerbe} the mapping
%  \[
%  (\shE,V,\lambda)\to (\shE,F^{-i*}_kV, F^{-i*}_k\lambda)
%  \]
%  is an equivalence between $\shD_i$ and the category $\shD_i(\stX)$ of triples $(\shG,W,\delta)$ where $\shG\in \Vect(\stX)$, $W\in \Vect(k)$ and $\delta$ is an isomorphism
%  \[
%  \delta\colon F^{i*}\shG \to 
%  \]
% \end{rmk}

\begin{rmk}\label{neutralization of local Tannakian categories}
  Using the same notations from 
  \ref{Tannakian interpretation of the local nori gerbe} and assuming that $k$
  is perfect, the functors $\shD_n \arr \Vect(k)$, 
  $(\shF,V,\lambda)\longmapsto F_k^{-n^*}V$, where $F_k$ is the absolute Frobenius of $k$, are compatible when $n$ varies, so they induce a functor $\shD_\infty\arr \Vect(k)$. It is easy to check that this functor is $k$-linear, exact and tensorial. In particular $\shD_\infty$ has a neutralization or, in other words
$\Pi^\LL_{\stX/k}(k)\neq \varnothing$.
\end{rmk}

In fact there is much more: over a perfect field a pro-local gerbe is neutral, and the neutralization is unique up to a unique isomorphism. 

\begin{lem} \label{profinite gerbe} 
Let $\Gamma$ be a pro-local gerbe over a perfect field $k$.
Then $\Gamma(k)$ is equivalent to a set with one point, 
in other words, it is a non-empty groupoid in which between
every two objects there exists exactly one isomorphism.
Equivalently, the Tannakian category $\Vect(\Gamma)$ has
a neutral fiber functor which is unique up to a unique isomorphism.
\end{lem}
\begin{proof}
Since $\Gamma$ is a profinite gerbe, we may write
$\Gamma:=\varprojlim_{i\in I}\Gamma_i$, where $I$ is a cofiltered
essentially small category and $\Gamma_i$ are finite and local
gerbes over $k$. In this way we reduce to the case where $\Gamma$ is finite.

We first show that $\Gamma(k)\neq \varnothing$. The stack $\Gamma$
is reduced because it has a faithfully flat map from a reduced
scheme, namely the spectrum of some field. Moreover since
$\Gamma$ is local we clearly have $\Gamma=\Pi^\LL_{\Gamma/k}$.
Thus $\Gamma(k)=\Pi^\LL_{\stX/k}(k)\neq \varnothing$
by \ref{neutralization of local Tannakian categories}.

In particular $\Gamma=\Bi G$, where $G$ is a finite and local
group scheme over $k$. If $P$ is a $G$-torsor over $k$ then
$P$ is finite and geometrically connected. Thus $P=\Spec A$,
where~$A$ is local, finite with purely inseparable residue
field extension $l/k$. Since $k$ is perfect, we have $l=k$.
Thus $P(k)$ consists of one element, as we wanted to show.
\end{proof}

\begin{rmk} The key point of Lemma \ref{profinite gerbe} is that
    any non-empty finite
    and
    local stack over a perfect field $k$ has a $k$-rational point. Indeed, consider the $n$-th relative Frobenius twist: $\Gamma\arr \Gamma^{(n)}$.  According to \cite[Lemma 3.6, p.~11]{TZ19} there exists $n\in \N$ such that the Frobenius twist factors through $\Gamma_\et$ which is equal to $\Spec(k)$ because $\Gamma$ is finite and local. This provides a rational section for $\Gamma^{(n)}$. But since $k$ is perfect, by twisting back with $(-)^{(-n)}$, we see that $\Gamma$ has a $k$-rational point.
\end{rmk}

Since affine group schemes are the same as affine gerbes with
a given rational section, we obtain the following:

\begin{cor}\label{local groups are local gerbes} Let $k$ be a perfect field.
The functor
   \[
  \begin{tikzcd}[column sep=20,
    ,row sep = 0ex
    ,/tikz/column 1/.append style={anchor=base east}
    ,/tikz/column 2/.append style={anchor=base west}
    ]
\{\textup{pro-local group schemes over }k\} \ar[r]
& \{\textup{pro-local gerbes over }k\} \\
G \ar[r,mapsto] & \Bi G
  \end{tikzcd}
  \]
is an equivalence of categories (meaning that between two functors
of pro-local gerbes there exists at most one natural isomorphism).
\end{cor}

\begin{defn} \label{def local Nori}
 Let $\stX$ be a reduced algebraic stack over a perfect field $k$.
We denote by $\pi^\L(\stX/k)$ the pro-local
group scheme such that $\Pi^\L_{\stX/k}=\Bi \pi^\L(\stX/k)$
and call it the \emph{local Nori fundamental group scheme} of
$\stX$ over $k$.
\end{defn}

\begin{defn}
 If $G$ is an affine group scheme over $k$, a $G$-torsor
$P\to X$ is \emph{minimal} if it is not induced by a torsor
under a strict subgroup of $G$. 
\end{defn}

\begin{rmk}
 This notion of minimality for a torsor is similar but different from the Nori reduced one (see \cite[Definition 5.10]{BV15}). Indeed for a $G$-torsor $P\to X$ being Nori reduced means that for any factorization 
 \[
 X\to \Gamma \arrdi \gamma \Bi G
 \]
 where $\Gamma$ is a finite gerbe and $\gamma$ is faithful, the map $\gamma$ is an isomorphism. Minimality requires the same property, but under the additional assumption that $\Gamma\simeq \Bi H$ and $\gamma$ "preserves the trivial torsors", that is, it is induced by a homomorphism $H\to G$.
 
 We see therefore that Nori reduced implies minimal, but the converse is not true in general: if $L/k$ is a Galois extension with group $G$ then $\Spec L\to \Spec k$ is a minimal $G$-torsor while the corresponding map $\Spec k\to \Bi G$ is not Nori reduced.
 
 On the other hand the two notions coincide in the following two cases.
 \begin{itemize}
  \item The field $k$ is algebraically closed, because in this case for any finite gerbe $\Gamma$ the category $\Gamma(k)$ is non-empty, so that $\Gamma\simeq \Bi H$ is neutral, and any $H$-torsor over $k$ is trivial. 
%   : if $X\arrdi
%       \alpha \Gamma \arrdi \gamma \Bi G$ is a factorization
%       through a finite gerbe then  $\Gamma$ is neutral,  and for any
%       neutralization $p\in\Gamma(k)$, the image $\gamma(p)$  is isomorphic to the
%       neutralization of $\Bi G$ because $(\Bi G)(k)$ is
%        the
%       category of the  trivial
%       $G$-torsors over $k$.
  \item The group $G$ is a pro-local group scheme, thanks to Corollary \ref{local groups are local gerbes}.
 \end{itemize}
\end{rmk}

\begin{prop}\label{prop: universal property local fund group}
 Let $\stX$ be a reduced algebraic stack over a perfect field $k$.
Then $\pi^\L(\stX/k)$ is the unique pro-local fundamental group
scheme over $k$ with natural equivalences
 \[
\Hom_k(\pi^\LL(\stX/k),G) \arr \{G\text{-torsors over }\stX\} 
\]
functorial in the finite and local group scheme $G$ over $k$.
Moreover a group homomorphism
\[
\pi^\LL(\stX/k)\arr G
\]
is surjective if and only if the corresponding $G$-torsor is minimal.
\end{prop}

\begin{proof}
 This follows from the universal property of the local Nori gerbe
and the equivalence in \ref{local groups are local gerbes}. In
particular the category on the right is indeed a set, that is
there exists at most one isomorphism between two $G$-torsors over $\stX$.
\end{proof}

\begin{rmk}\label{rmk: universal property local fund group}
 An alternative way to state \ref{prop: universal property local fund group}
is that there are natural bijections
 \[
 \Hom_k(\pi^\LL(\stX/k),G) \simeq \Hl^1(\stX,G)
 \]
functorial in the finite, local group scheme $G$ over $k$.
\end{rmk}

In this special situation of local group schemes there is no need
for choosing a rational or geometric point. A similar phenomenon
also appears in  \cite[Proposition 2.21 (ii) and Remark 2.22]{Zh18}.
The following proposition shows that this is indeed not a coincidence:

\begin{prop}\label{projlim}
 Let $\stX$ be a reduced algebraic stack over a perfect field $k$ and consider the category 
 $\shN(\stX/k)$ of pairs $(G,\stP)$ where $G$ is a finite and local group
 scheme over $k$ and $f\colon\stP\arr\stX$ is a $G$-torsor.
 Then $\shN(\stX/k)$ is a small cofiltered category and there is
 a canonical isomorphism
 $$
 \pi^\LL(\stX/k) \simeq \varprojlim_{(G,\shP)\in \shN(\stX/k)} G.
 $$
\end{prop}

\begin{proof}
By the above discussion we obtain that the category $\shN(\stX/k)$
is equivalent to the category $\Hom_k(\pi^\LL(\stX/k),-)$ of morphisms
from $\pi^\LL(\stX/k)$ to finite and local group schemes. Notice that
$\Hom_k(\pi^\LL(\stX/k),-)$ has fiber products and, in particular, 
it is cofiltered (see \cite[Remark 1.3, (i)]{Zh18}). Moreover,
if $\sN'(\sX/k)$ is the full subcategory of $\Hom_k(\pi^\LL(\stX/k),-)$
consisting of quotient maps, then we have an isomorphism
\[
\varprojlim_{\shN(\stX/k)} G \arrdi \simeq
\varprojlim_{\shN'(\stX/k)} G
\]
which is the limit of all finite and local quotients of
$\pi^\LL(\stX/k)$. Thus the Hopf algebra of
$\varprojlim_{\shN(\stX/k)} G$ is contained in
$k[\pi^\LL(\stX/k)]$. Since $\pi^\LL(\stX/k)$ is pro-local,
 it is a cofiltered limit of some of its finite and local quotients.
In this way we obtain an inclusion of Hopf algebras in the other
direction, and this finishes the proof.
\end{proof}

\begin{cor}\label{Zhang}
If $X$ is a reduced scheme over a perfect field $k$ with a
geometric point $x\colon \Spec \Omega\arr X$, where $\Omega$
is an algebraically closed field, then $\pi^\LL(X/k)$ coincides
with the group scheme $\pi^\LL(X/k,x)$ defined in
\cite[Definition 4.5(iv)]{Zh18}.
\end{cor}

\begin{proof}
The right-hand side of the isomorphism in~\ref{projlim}
is the group scheme $\pi^\LL(X/k,x)$ as defined in
\cite[Definitions 3.6 and 4.5(iv)]{Zh18}.
\end{proof}

\begin{prop}
If $K/k$ is a field extension with $k$ perfect, then we have an
isomorphism between $\pi^\LL(K/k)$ as defined in~\ref{Local Tanna}
and $\pi^\LL(\Spec K/k)$ as defined in~\ref{def local Nori}.
\end{prop}

\begin{proof}
 Using notations from \ref{Tannakian interpretation of the local nori gerbe}
 for $\stX=\Spec K$ we have to show that there is an equivalence of Tannakian
 categories $\shD_\infty \simeq \shD_\infty(K/k)$. 
If $(V,W,\phi)\in\shD_n(K/k)$ with $V\in \Vect(K)$, $W\in \Vect( k)$ and
$\phi\colon F^{n^*}V\simeq W\otimes_k K$, then via the isomorphism of fields
$K^{1/p^n}\xrightarrow{} K$, $x\mapsto x^{p^n}$,
we get
$F^{n^*}V\simeq V\otimes_K K^{1/p^n}$ and
$W\otimes_k K\simeq (F_k^{-n^*}W)\otimes_k K^{1/p^n}$. 
Then it is not difficult to show that, for $n\in\N$, the category 
$\shD_n(K/k)$ is equivalent to the category of triples 
$(M,N,\varphi)$ where 
$M\in\Vect(K)$, $N\in \Vect(k)$ and 
$\varphi\colon M\otimes_K K^{1/p^n}\simeq N\otimes_k K^{1/p^n}$
is an isomorphism. By passing to the limit we get $\shD_\infty \simeq \shD_\infty(K/k)$.
\end{proof}

We conclude this section by showing that the local fundamental group
detects isomorphisms. We need the following lemma.

\begin{lem}\label{picard and extensions}
Let $K/k$ be an extension of fields. Then there are natural
isomorphisms:
\[
\begin{array}{lcl}
\Hom(\pi^\L(K/k),\alpha_{p^n})\simeq K/K^{p^n} & \text{and} &
\Hom(\pi^\L(K/k),\mathbb{G}_a) \simeq K^\perf/K, \\
\Hom(\pi^\L(K/k),\mu_{p^n})\simeq K^*/{K^*}^{p^n} & \text{and} &
\Hom(\pi^\L(K/k),\Gm) \simeq (K^\perf)^*/K^*.
\end{array}
 \]
\end{lem}

\begin{proof}
The isomorphism for $\alpha_{p^n}$ follows
from \ref{rmk: universal property local fund group}. Moreover, the
isomorphisms for $n$ and for $n+1$ fit in a commutative diagram:
\[
\begin{tikzcd}
\Hom(\pi^\L(K/k),\alpha_{p^n}) \ar[r,"\cong"] \ar[d,hook]
& K/K^{p^n} \ar[r,"\cong"] \ar[d,swap,"x\mapsto x^p"] & K^{1/p^n}/K \ar[d,hook] \\
\Hom(\pi^\L(K/k),\alpha_{p^{n+1}}) \ar[r,"\cong"]
& K/K^{p^{n+1}} \ar[r,"\cong"] & K^{1/p^{n+1}}/K.
\end{tikzcd}
\]
Since $\pi^\L(K/k)$ is pro-local, we have:
\[
\Hom(\pi^\L(K/k),\mathbb{G}_a) \simeq \varinjlim_n \Hom(\pi^\L(K/k),\alpha_{p^n}).
\]
The transition maps are described in the above commutative diagram,
hence
\[
 \Hom(\pi^\L(K/k),\mathbb{G}_a) \simeq \varinjlim_n K^{1/p^n}/K = K^\perf/K.
 \]
In a similar way, one retrieves the isomorphism for $\mu_{p^n}$ from
\ref{rmk: universal property local fund group} and derives
the isomorphism for $\Gm$ (which was alternatively obtained in
\ref{picard of local gerbe}).
\end{proof}
 
 \begin{prop}\label{isom fund group}
 Let $k$ be a perfect field and $E/K$ be a finitely generated
extension of fields over $k$. Then the map
 \[
 \pi^\L(E/k)\to \pi^\L(K/k)
 \]
 is an isomorphism if and only if $K=E$ or $E/K$ is a finite
extension of perfect fields.
\end{prop}

\begin{proof}
The if part follows from \ref{perfect means group zero}.
 Assume that the map on fundamental groups is an isomorphism. By \ref{picard and extensions} the map
 \begin{equation}\label{perf iso}
 K^\perf/K\to E^\perf/E
 \end{equation}
 is an isomorphism.
 In particular the map $K^\perf \to E^\perf/E$ is surjective,
and $E^\perf=K^\perf(E)$. Write $K\subseteq F\subseteq E$
with $F/K$ purely transcendental and $E/F$ finite. We have
 \[
   K^\perf \subseteq K^\perf(F) \subseteq K^\perf(E)=E^\perf.
 \]
Since  the finitely generated extension 
$K^\perf(E)=E^\perf$ over $K^\perf$ contains the extension $K^\perf(F)/K^\perf$, $K^\perf(F)$
cannot contain indeterminates, that is $F=K$ and $E/K$ is finite.
 
 Now split $E/K$ as $K\subseteq S\subseteq E$ with $S/K$ separable and $E/S$ purely inseparable. Let $n$ be an index such that $E^{p^n}\subseteq S$. By \ref{picard and extensions} the composition
 \[
 K/K^{p^n}\to S/S^{p^n}\to E/E^{p^n}
 \]
 is an isomorphism. In particular the second map is surjective. But, since $E^{p^n}\subseteq S$, the image of the second map is $S/E^{p^n}$ and therefore $S=E$, that is $E/K$ is separable.
 
 So $E=K[x]/(f(x))$, for some separable polynomial $f\in K[x]$. The separability implies that 
 \[
 E^\perf = K^\perf[x]/(f(x))
 \]
 so that the isomorphism (\ref{perf iso}) is the inclusion
 \[
 K^\perf/K \to (K^\perf/K)^{ \deg f}.
 \]
 It follows that either $K=K^\perf$, so that also $E$ is perfect,
or $\deg f=1$, that is $K=E$.
\end{proof}

\section{Generic surjectivity} \label{section:generic surjectivity}

We fix a perfect field $k$ of characteristic $p>0$. The aim of
this section is to show Theorem \ref{main theorem II} and
related results. 

Let $\sX$ be an algebraic stack over $k$, and let $G$ be a
finite local $k$-group scheme.
Firstly, we look for a criterion for the surjectivity of the map induced on local fundamental group schemes.

\begin{rmk}\label{section and reduce a torsor to a group}
    If $\sP\arr \sX$ is a $G$-torsor over an algebraic stack
    $\sX$ and
    $H\subseteq G$ is a subgroup then $\sP$ is induced by an
            $H$-torsor if and only if $\sP/H \arr \sX$ has a
            section.
            
                Indeed, if $\sP/H\arr \sX$ admits a
            section, then the $H$-torsor inducing
            $\sP\arr\sX$ is just the pullback of $\sP\arr\sP/H$
            along the section. Conversely, if $\sP\arr \sX$
            reduces to a $H$-torsor $\sQ\arr\sX$, then the map
            $$\sX\simeq \sQ/H\arr \sP/H\arr \sX$$ provides a
            section.
\end{rmk}

\begin{rmk}\label{section and reduceness}
If $\stX$ is a reduced algebraic stack over $k$, $\stP\to \stX$ is a torsor under a finite local group scheme $G$ over $k$ and $H\subseteq G$ is a subgroup then $\stP/H\to \stX$ is a finite universal homeomorphism. In particular $\stP\to \stX$ is induced by an $H$-torsor if and only if $(\stP/H)_\red \to \stX$ is an isomorphism.
\end{rmk}

%\begin{defn}
% A $G$-torsor $\stP\to \stX$ is called minimal if it is not induced by a proper subgroup of $G$.
%\end{defn}

\begin{rmk} \label{surjectivity and minimal torsor}
If $\sV\arr \sX$ is a map between reduced algebraic stacks then
$\pi^\LL(\sV/k)\arr\pi^\LL(\sX/k)$
is surjective if and only if the following condition holds: if
$\sP\arr \sX$ is a minimal $G$-torsor for a local group scheme $G$ then
$\sP\times_\sX\sV\arr \sV$ is also minimal. 

Indeed, the condition that the
pullback of any minimal finite local $G$-torsor on $\sX$ to $\sV$ is 
minimal is equivalent to the condition that any surjective map
$\pi^\LL(\sX/k)\arr G$ is still surjective after composing with
$\pi^\LL(\sV/k)\arr\pi^\LL(\sX/k)$, and this is
equivalent to saying that the map
$\pi^\LL(\sV/k)\arr\pi^\LL(\sX/k)$ itself is surjective.
\end{rmk}

Putting together \ref{section and reduce a torsor to a group} and \ref{surjectivity and minimal torsor} we obtain the following criterion.
\begin{lem}\label{criterion surjective}
 Let $\sV\arr \sX$ be a map between reduced algebraic stacks. Assume that if $\stP\to \stX$ is a minimal torsor under a finite local group scheme $G$ over $k$ and $H\subseteq G$ is a subgroup then any map $ \stV\to \stP/H$ over $\stX$ extends to a section of $\stP/H\to \stX$.  Then
 $$
 \pi^\L(\stV/k)\to \pi^\L(\stX/k)
 $$ is surjective. 
\end{lem}

\begin{lem}\label{surjective fund group}
    Let $K$ be a field extension of $k$.
Let $E$ be  a non-zero  $K$-algebra
such that if $e\in E$ and $e^p\in K$ then $e\in K$
(e.g. when $E$ is geometrically reduced over $K$). Then the map 
 \[
 \pi^\L(E/k)\to \pi^\L(K/k)
 \]
 is surjective. 
\end{lem}

\begin{proof}
The condition on $E$ implies that $E$ is reduced, so that   $\pi^\L(E/k)$ exists by \ref{existence of local gerbe}.  

Following \ref{criterion surjective} let $P\to \Spec K$ be a minimal $G$-torsor with a map $\Spec E\to P/H$ over $K$. By \ref{section and reduceness} we see that $(P/H)_\red \to \Spec K$ is a finite universal homeomorphism, that is $(P/H)_\red=\Spec F$ for a finite purely inseparable field extension $F/K$, and we need to show that it is an  isomorphism.

Since $E$ is reduced we have factorizations $\Spec E \to (P/H)_\red \subseteq P/H$ and therefore  
$$K \subseteq F\subseteq E$$ 
The condition on $E$ implies  $F=K$ as required.
 \end{proof}

\begin{cor}\label{series surjective}
 The map
 \[
 \pi^\L(k((t))/k)\to \pi^\L(k(t)/k)
 \]
 is surjective.
\end{cor}

\begin{proof}
From \ref{surjective fund group} we just have to show that if
$u\in k((t))$ and $u^p\in k(t)$, then $u\in k(t)$. Multiplying
$u$ with a denominator of $u^p$ we may assume that $u^p\in
k[t]$. Now  $u=\sum u_it^i$ is a Laurent
series,
and $u_i^p=0$ for $i\gg 0$, hence $u_i=0$ for $i\gg 0$, hence
$u$ is a rational function.
\end{proof}

\begin{prop}\label{one variable is enough}
Let  $E/k$ be a finitely generated extension
of transcendence degree $n$. Then there are elements $t_1,\dots,t_n\in E$
algebraically independent over~$k$ such that
 \[
 \pi^\L(E/k)\arr \pi^\L(k(t_1,\dots,t_n)/k)
 \]
 is surjective. In particular, if $E/k$ is not finite, there exists
an indeterminate $t\in E$ and a surjective map
 \[
 \pi^\L(E/k)\to \pi^\L(k(t)/k)
 \]
\end{prop}

\begin{proof}
By \cite[\href{http://stacks.math.columbia.edu/tag/030Q}{030Q}]{SP21}
there is a subfield $k\subseteq F \subseteq E$
such that $F/k$ is purely transcental of degree $n$ and $E/F$ is finite
and separable. We see that both extensions $E/F$ and $F/k(t_1)$ satisfy
the hypothesis of \ref{surjective fund group}, proving the result.
\end{proof}

 \begin{lem}\label{geometrically reduced map}
 Let $f:\sY\arr\sX$ be a faithfully flat geometrically reduced map of
 reduced algebraic stacks over $k$. Then the induced map
 $$\pi^\L(\sY/k)\to \pi^\L(\sX/k)$$ is surjective.
 \end{lem}
 \begin{proof}
     Following \ref{criterion surjective} let  $\sP\to \sX$ be a
minimal $G$-torsor with a map $\stY\to \stP/H$ over $\stX$.

Since $f$ is geometrically reduced and thanks to \ref{section and reduceness} we have that
$$
((\sP/H)\times_\sX\sY)_\red=(\sP/H)_\red\times_\sX\sY \to \stY
$$
is an isomorphism. In other words the map $(\sP/H)_\red\to \sX$ is an
isomorphism after pulling back to $\sY$ and, again by \ref{section and
reduceness}, we need to show that is an isomorphism.

For this, we can replace $\stX$ by an affine open of an atlas and, since $(\sP/H)_\red\arr\sX$ is affine, also assume it is the spectrum of a local ring.  In this case $\stY\arr \stX$ is automatically an fpqc covering, and by descent we get the result.
 \end{proof}

\begin{lem}\label{affinazation surjective}
 Let $\stX$ be a reduced algebraic stack over $k$. Then
 \[
 \pi^\L(\stX/k)\to \pi^\L( \Hl^0(\odi\stX)/k)
 \]
 is surjective.
\end{lem}
\begin{proof}
    Following \ref{criterion surjective} let $P \arr \Spec \Hl^0(\odi \stX)$ be a minimal $G$-torsor with a map $\stX\to P/H$ over $\Hl^0(\odi\stX)$. The desired factorization exists because $P/H$ is affine.
\end{proof}

In the following, we will deal with normal algebraic stacks. These are  algebraic stacks which admit a smooth atlas from a normal scheme in the sense of \cite[\href{https://stacks.math.columbia.edu/tag/033H}{033H}]{SP21}.  In particular, these algebraic stacks may not be locally Noetherian.
Notice also that any smooth atlas of a normal algebraic stack is normal (see \cite[\href{https://stacks.math.columbia.edu/tag/04YH}{04YH}]{SP21}). 

We need a technical lemma.
\begin{lem} \label{atlas and generic points}
 Let $\stX$ be a quasi-separated algebraic stack and $f\colon V \to \stX$ be an open and locally of finite type map from a scheme. Then:
 \begin{enumerate}
  \item the morphism $f$ maps generic points into generic points;
  \item if $\sX$ is irreducible then all generic points of $V$ have an open irreducible neighborhood;
  \item if $V$ is quasi-compact and $\stZ$ is an irreducible component of $\stX$ then $f^{-1}(\stZ)$ is either empty or a finite union of irreducible components $W$ of $V$ such that $W\to \stZ$ is dominant. 
 \end{enumerate}
\end{lem}
\begin{proof}
Notice that the topological space $|\stX|$ is sober by
\cite[\href{https://stacks.math.columbia.edu/tag/0DQQ}{0DQQ}]{SP21}.

$(3) \then (1)$. Let $v\in V$ be a generic point and $\xi\in |\stX|$ a generic point such that $f(v)\in \overline{\{\xi\}}$. Replacing $V$ by a quasi-compact open neighborhood of $v$ and applying $(3)$, we can conclude that the map $\overline{\{v\}} \to \overline{\{\xi\}}$ is dominant, which means $f(v)=\xi$

$(3) \then (2)$. Let $v\in V$ be a generic point and $U\subseteq V$ a quasi-compact open subset such that $v\in U$. By $(3)$ applied to $\stZ=\stX$, it follows that $U$ has finitely many irreducible components. Thus it is enough to remove from $U$ all the irreducible components not containing $v$.

$(3)$. As $\stX$ is quasi-separated and $V$ is quasi-compact it follows that $f\colon V\to \stX$ is a quasi-compact map. We can assume $f^{-1}(\stZ)\neq \emptyset$. Let $\xi$ be the generic point of $\stZ$ and set $T=f^{-1}(\xi)$.

We have $\overline{T}=f^{-1}(\stZ)$, so that, in particular, $T\neq \emptyset$. Indeed otherwise we would have the contradiction
\[
f(V-\overline{T})\cap \stZ\neq \emptyset \then \xi \in f(V-\overline{T}) \then T\cap (V-\overline{T})\neq \emptyset
\]
Here we used that $f(V-\overline{T})$ is open as $f$ is open.

If $\Spec \Omega\to \stX$ is a geometric point mapping to $\xi$, the continuous map $V_\Omega=\Spec \Omega \times_\stX V \to V$ surjects onto $T$. As $f$ is quasi-compact and locally of finite type, it follows that $V_\Omega$ is of finite type over $\Omega$. If $t_1,\dots,t_r$ are the image of the (finitely many) generic points of $V_\Omega$ along the map $V_\Omega\to T$, then 
$$
T=\overline{\{t_1\}}\cup \cdots \cup \overline{\{t_r\}}
$$
where the closure is taken inside the topological space $T$. Taking the closure now inside $V$ we can conclude that the generic points of $\overline T=f^{-1}(\stZ)$ are among the $t_1,\dots,t_r$. So they are finitely many and they all maps to the generic point $\xi$ as required.
\end{proof}

 \begin{lem}\label{normal algebraic stack}
 Let $\sX$ be a normal, quasi-separated and irreducible algebraic stack over $k$, and
 let $\sU\subseteq \sX$ be a non-empty open substack. Then the
 map $$\pi^\LL(\sU/k)\arr\pi^\LL(\sX/k)$$ is surjective.
 \end{lem}
 \begin{proof}
 Following \ref{criterion surjective} let $\sP\to\sX$ be a minimal $G$-torsor with a section $\stU\to \stP/H$ over $\stX$. By \ref{section and reduceness} the morphism $(\stP/H)_\red \to \stX$ is an isomorphism over $\stU$ and we need to show that $(\stP/H)_\red \to \stX$ is an isomorphism as well.
 
 We show that a finite birational map $\stY\to \stX$ from a reduced algebraic stack is an isomorphism. We can assume that $\stX$ is quasi-compact and consider a smooth atlas $V\to \stX$ from a quasi-compact scheme.
     
     By \ref{atlas and generic points}, $(3)$ the scheme $V$ has finitely many irreducible components. Since $V$ is normal, it follows that it is a finite disjoint union of integral normal schemes. Thus we can assume that $\stX$ is an integral normal scheme and also that it is affine. Thus $\stX=\Spec D$, for a normal domain $D$ and $\stY=\Spec B$, for a reduced ring $B$.
     
     Since $\stY$ contains a dense open subset isomorphic to an open subset of $\stX$ and therefore irreducible, it follows that $\stY$ is integral. More precisely that $B$ is a domain with the same fraction field of $D$. As $D\to B$ is an (injective) integral extension and $D$ is normal we can conclude that $D=B$.
     \end{proof}

 \begin{lem}\label{pre-generic surjectivity}
    Let $X$ be an integral normal scheme over $k$, and let
    $\Spec\Omega \arr X$ be the generic point of $X$. Then the
    map $$\pi^\LL(\Spec \Omega/k)\arr\pi^\LL(X/k)$$ is surjective.
 \end{lem}
 \begin{proof}
 Following \ref{criterion surjective} let $P\to X$ be a minimal $G$-torsor with a map $\Spec \Omega\to P/H$ over $X$.
 By \ref{section and reduceness} it follows that $(P/H)_\red \to X$ is a finite universal homeomorphism which is generically an isomorphism. It is therefore an isomorphism: one first reduce to the affine case and then argue as in the end of the proof of \ref{normal algebraic stack}.
  \end{proof}

\begin{lem}\label{generic surjectivity}
Let $\sX$ be a normal, quasi-separated, irreducible algebraic stack over $k$. Let
$\Spec\Omega$ be a generic point of a smooth atlas of $\sX$,
then the map $$\pi^\LL(\Spec\Omega/k)\arr\pi^\LL(\sX/k)$$ is
surjective. 
\end{lem}

\begin{proof}
Let  $f\colon V\arr\sX$ be the smooth atlas. By \ref{atlas and generic points}, $(2)$ there is an open and irreducible subset $U$ of $V$ containing the given generic point. In particular $U$ is integral and normal. Let $\sU=f(U) \subseteq \sX$.  Then we can decompose
$\pi^\LL(\Spec\Omega/k)\arr\pi^\LL(\sX/k)$ into
$$\pi^\LL(\Spec\Omega/k)\xrightarrow{\ \lambda_1\
}\pi^\LL(U/k)\xrightarrow{\ \lambda_2\
}\pi^\LL(\sU/k)\xrightarrow{\ \lambda_3\ }\pi^\LL(\sX/k)$$
where $\lambda_1$ is surjective because of \ref{pre-generic
surjectivity},
$\lambda_2$ is surjective by \ref{geometrically reduced map} because $U\to \sU$ is geometrically
reduced and faithfully flat, and $\lambda_3$ is surjective by \ref{normal algebraic stack}, whence the result.
\end{proof}

\begin{thm}\label{main surjectivity fund group}
 Let $\stX$ be a normal, quasi-separated and irreducible algebraic stack over $k$, $\stV$ a reduced algebraic stack and $\stV\to \stX$ be a map. Assume that there is a commutative diagram
   \[
  \begin{tikzpicture}[xscale=1.5,yscale=-1.2]
    \node (A0_0) at (0, 0) {$\stW$};
    \node (A0_1) at (1, 0) {$\stV$};
    \node (A1_0) at (0, 1) {$\Spec \Omega$};
    \node (A1_1) at (1, 1) {$\stX$};
    \path (A0_0) edge [->]node [auto] {$\scriptstyle{}$} (A0_1);
    \path (A0_0) edge [->]node [auto] {$\scriptstyle{}$} (A1_0);
    \path (A0_1) edge [->]node [auto] {$\scriptstyle{}$} (A1_1);
    \path (A1_0) edge [->]node [auto] {$\scriptstyle{}$} (A1_1);
  \end{tikzpicture}
  \]
where $\Spec \Omega$ is a generic point of a smooth atlas of $\stX$ and $\stW$ is a (non empty) reduced algebraic stack with the following property: if $z\in \Hl^0(\odi \stW)$ and $z^p \in \Omega$ then $z\in \Omega$. Then the map
 \[
 \pi^\L(\stV/k)\to \pi^\L(\stX/k)
 \]
 is surjective.
\end{thm}

\begin{proof} The map $\pi^\LL(\Spec
    \Omega/k)\arr \pi^\LL(\sX/k)$ is surjective by \ref{generic
    surjectivity}. By functoriality
    it is enough to show that $ \pi^\L(\stW/k)\to
    \pi^\L(\Spec\Omega/k)$ is surjective. This map factors as
    $$ \pi^\L(\stW/k)\xrightarrow{\psi}
    \pi^\LL(\Spec\Hl^0(\sO_\stW)/k)\xrightarrow{\phi}
    \pi^\L(\Spec\Omega/k)$$ 
    Since $\phi$ is surjective by \ref{surjective fund group}, while $\psi$ is surjective by \ref{affinazation surjective}, we get the result.
 \end{proof}

 \begin{rmk}
  Theorem \ref{main surjectivity fund group} continues to be true if $\stX=X$ is a normal integral scheme and $\Spec \Omega\to X$ is its generic point. Just use \ref{pre-generic surjectivity} instead of \ref{generic surjectivity}.
 \end{rmk}

\begin{proof}[Proof of Theorem \ref{main theorem II}]
  Taking into account \ref{atlas and generic points}, $(1)$, everything follows from \ref{main surjectivity fund group}.
\end{proof}

\end{document}